\documentclass[11pt,reqno,UTF8,twoside]{amsart}
	\usepackage{mathrsfs}
	\usepackage{amsfonts,amssymb,amsmath,amsthm}
	\usepackage{cite}
	\usepackage[colorlinks=true,citecolor=red]{hyperref}
	\usepackage{titletoc}
	\usepackage{geometry}
	\usepackage{bm}
	\usepackage{indentfirst}
	\usepackage{graphicx}
	\usepackage{float} 
	\usepackage{booktabs}
	\usepackage{longtable}
	\usepackage{subfigure}
	\usepackage{stmaryrd}
	\usepackage{enumerate}
	\usepackage{tikz}
	\usepackage{ulem}
	\usepackage{color,soul}
	\usepackage{setspace}
	\usepackage{stmaryrd}
	\usepackage[pagewise]{lineno} 
	\allowdisplaybreaks[2]
	\usepackage{appendix}
	\usepackage{cases}
	\usepackage{todonotes}
	\usepackage{enumitem}

	\numberwithin{equation}{section}
	\newtheorem{theorem}{Theorem}[section]
	\newtheorem{lemma}[theorem]{Lemma}
	
	\newtheorem{proposition}[theorem]{Proposition}
	\newtheorem{maintheorem}{Theorem}
	
	\newtheorem*{conjecture}{Conjecture}

    \theoremstyle{definition}
	\newtheorem{remark}[theorem]{Remark}
	\newtheorem{definition}[theorem]{Definition}
	\newtheorem{example}[theorem]{Example}

	\newcommand{\T}{\mathbb{T}}
	\newcommand{\N}{\mathbb{N}}
	\newcommand{\R}{\mathbb{R}}
	\newcommand{\C}{\mathbb{C}}
	\newcommand{\Z}{\mathbb{Z}}
	\newcommand{\E}{\mathbb{E}}
	
	\newcommand{\Pb}{\mathbb{P}}

	\newcommand{\dvg}{\operatorname{div}}

	\newcommand{\V}{\mathcal{V}}
	\hypersetup{linkcolor=blue}

	\makeatletter
	\@namedef{subjclassname@2020}{\textup{2020} Mathematics Subject
	Classification}
	\makeatother

\begin{document}

	\title[ Invariant Gaussian measures for 2D Euler ]
	{{\Large O{\MakeLowercase{n the classification of invariant} G{\MakeLowercase{aussian measures for} \\
	\vspace{1mm}
	\MakeLowercase{the 2}D E{\MakeLowercase{uler equation}  
	}}}
	}}

	\author[Z. Liu]{ {\small Z\MakeLowercase{iyu} L\MakeLowercase{iu}}}

	\address[Ziyu Liu]{School of Mathematics and Physics, University of Science and Technology Beijing, 100083, Beijing, China.}
	\email{ziyu@ustb.edu.cn}

	\begin{abstract}
        
    We consider the two-dimensional incompressible Euler equation on $\T^2$ with Gaussian random initial data having independent Fourier coefficients. For every $\sigma>0$, we prove that such a Gaussian measure on $H^\sigma(\T^2)$ is invariant under the Euler flow if and only if it is supported either on shear flows or on cellular flows. This settles the invariant-measure classification conjecture posed by Bedrossian and Latocca ({\it Ann. Inst. H. Poincar\'e C Anal. Non Lin\'eaire}, 2026). The proof relies on closure of the Fourier support under non-degenerate interactions and an affine relation among the inverse variances along Euler triples. 
	\end{abstract}



	\maketitle
	\setcounter{tocdepth}{1}
	\tableofcontents

	\section{Introduction}\label{Sec 1}

	\subsection{Background and motivation}\label{Sec 1.1}
    Random initial data have become an important tool in nonlinear PDE, providing a framework for studying typical behavior and statistical properties beyond individual trajectories; see, e.g., \cite{Bourgain-94,BT-08,DNY-24}. In fluid dynamics, this viewpoint is closely connected with statistical hydrodynamics and invariant-measure methods for long-time behavior; see, e.g., \cite{Kuksin-04,Latocca-23}.

    \vspace{0.6em}
	In this paper, we consider the two-dimensional incompressible Euler equation posed on $\T^2=(\R/(2\pi\Z))^2$ in vorticity form with random initial data, which reads
	\begin{equation}\label{eq euler}
	\begin{cases}	\partial_t\Omega+u\cdot\nabla\Omega=0,\quad x\in\T^2,\;t\in\R,\\
	\Omega(0,\cdot)=\Omega_{in}(\cdot).
	\end{cases}
	\end{equation}
	Here $\Omega\colon\T^2\rightarrow\R$ is the scalar vorticity, and $u\colon\T^2\rightarrow\R^2$ denotes the velocity. Specifically, $u$ is related to $\Omega$ through the Biot--Savart law given by
	\begin{equation*}	
	u=\nabla^\perp(-\Delta)^{-1}\Omega,\quad 
	\Omega=\nabla\wedge u=\partial_1u_2-\partial_2u_1,
	\end{equation*}
	where  $\nabla^\perp=(-\partial_2,\partial_1)$.

	To investigate Gaussian measures preserved by the Euler dynamics, we consider Gaussian random initial vorticities of the form
	\begin{equation}\label{eq Omega0}
	\Omega_{in}^{\omega}(x):=\sum_{n\in\Z^2_*}a_ng_n^\omega e^{in\cdot x}.
	\end{equation}
	Here $\Z^2_*:=\Z^2\setminus\{0\}$, $\{g_n\}_{n\in\Z^2_*}$ are complex Gaussian random variables specified later, and $a=\{a_n\}_{n\in\Z^2_*}$ is a complex-valued sequence. We denote the law of $\Omega_{in}^\omega$ by $\mu_a$ and define its {\it active Fourier support} by
	\begin{equation*}
	\mathcal S(a):=\{n\in\Z^2_*:a_n\neq0\}.
	\end{equation*}

    On its active real Fourier subspace, the measure $\mu_a$ has the formal Gibbs-type representation
    \begin{equation*}
    d\mu_a(\Omega)=Z_a^{-1}\exp\bigl(-Q_a(\Omega)\bigr)\,d\Omega,\quad Q_a(\Omega)=\frac{1}{4}\sum_{n\in\mathcal S(a)}\frac{|\Omega_n|^2}{|a_n|^2},
    \end{equation*}
    where $d\Omega$ denotes the formal Lebesgue measure on the real Fourier subspace determined by $\Omega_{-n}=\overline{\Omega_n}$ and $Z_a$ is the formal normalizing constant. 
    
    For full Fourier support, the classical energy--enstrophy Gibbs family corresponds to
    \begin{equation*}
    |a_n|^{-2}=c_0+\frac{c_1}{|n|^2},
    \end{equation*}
    since enstrophy and kinetic energy have Fourier weights $1$ and $|n|^{-2}$, respectively. Such full-support Gibbs measures normally lie below positive Sobolev regularity. This leads to the classification problem considered here: which smoother Gaussian measures of the form \eqref{eq Omega0} can be invariant under the Euler flow, and how must their active Fourier supports be structured?

    \vspace{0.6em}
    
    In \cite{BL-26}, Bedrossian and Latocca first addressed this question. Under a rapid-decay assumption, they proved that the invariant measures of the form \eqref{eq Omega0} are exactly those supported on shear or cellular flows. For $\sigma>3$, they also showed that non-invariance is generic in the Baire-category sense. They further formulated the following conjecture for the full range $\sigma>0$; the corresponding quasi-invariance conjecture is stated in \cite[Conjecture~1.10]{BL-26}.

    \begin{conjecture}
    For every $\sigma>0$, the Gaussian measure $\mu_a$ induced by \eqref{eq Omega0} is invariant under the Euler flow if and only if its active Fourier support is contained either in a line through the origin or in a circle centered at the origin.
    \end{conjecture}

	In the present paper, we prove this conjecture by removing the rapid-decay assumption and establishing the complete classification throughout the full Sobolev range $\sigma>0$. We also prove that quasi-invariance is equivalent to invariance whenever the active Fourier support is contained in a finite union of lines through the origin.

	\subsection{Main results}\label{Sec 1.2}
    
    We now state the main results, beginning with the phase space on which the flow is defined.  For $\sigma>1$, classical Sobolev well-posedness theory \cite{Kato-67,MB-02} provides a global Euler flow $\Phi_t$ on $H^\sigma(\T^2)$. Meanwhile, for the full range $\sigma>0$, the random initial data possess additional almost-sure Hölder regularity. Specifically, for any $\alpha\in(0,\sigma\wedge1)$, $\Omega_{in}^\omega\in c_0^\alpha(\T^2)$ almost surely, where $c_0^\alpha(\T^2)$ denotes the little Hölder space; see Proposition~\ref{prop wp}. The Hölder well-posedness theory \cite{Wolibner-33,MB-02,MY-18} then yields a global Borel flow $\Phi_t$ on $c_0^\alpha(\T^2)$.

    We regard $\mu_a$ as a Borel probability measure on $c_0^\alpha(\T^2)$ and define invariance and quasi-invariance with respect to the Borel flow $\Phi_t\colon c_0^\alpha(\T^2)\to c_0^\alpha(\T^2)$. By uniqueness, the resulting notion is independent of the particular choice of $\alpha\in(0,\sigma\wedge1)$.

	\begin{definition}\label{def 1}
	A probability measure $\mu$ on $c_0^\alpha(\T^2)$ is called {\it invariant} under $\Phi_t$ if $(\Phi_t)_*\mu=\mu$ for any $t\in\R$. Accordingly, $\mu$ is called {\it quasi-invariant} under $\Phi_t$ if $(\Phi_t)_*\mu\sim\mu$ for any $t\in\R$.
	\end{definition}

	Next, to address our settings for $\Omega_{in}^\omega$, let $(\Xi,\mathcal F,\Pb)$ be a probability space and set
	\begin{equation*}
	\Z^2_+:=\{(n_1,n_2)\in\Z^2_*:n_1>0\text{ or }n_1=0,\;n_2>0\}.
	\end{equation*}    
	We will work under the following assumptions for $\{g_n\}_{n\in\Z^2_*}$ and $\{a_n\}_{n\in\Z^2_*}$ in \eqref{eq Omega0}.

	\begin{itemize}
	\item [$\mathbf{(A_1)}$] ({\bf Gaussian structure}) {\it 
	For any $n\in\Z^2_+$, let $g_n^\omega:=r_n^\omega+is_n^\omega$ and $g_{-n}^\omega:=\overline{g_n^\omega}$, where the family $\{r_n^\omega,s_n^\omega:n\in\Z^2_+\}$ consists of i.i.d real-valued standard Gaussian variables with law $\mathcal N(0,1)$ on $\Xi$. 
	}
    \end{itemize}
    
    As each $g_n$ is rotationally invariant in law, the law of the real-valued field depends only on $|a_n|$ on $\Z^2_+$. Thus, without loss of generality, we restrict to real-valued coefficients.

    \begin{itemize}   
	\item [$\mathbf{(A_2)}$] ({\bf Coefficient regularity}) {\it 
	For $\sigma>0$, the coefficients $a=\{a_n\}_{n\in\Z^2_*}$ satisfy
	\begin{equation*}
	a\in h^\sigma:=\left\{\{a_n\}_{n\in\Z^2_*}:a_n\in\R,\; a_{-n}=a_n,\;\|\{a_n\}\|_{h^\sigma}<\infty\right\},
	\end{equation*}
	where the norm $\|\cdot\|_{h^{\sigma}}$ is defined by $\|\{a_n\}\|_{h^\sigma}^2:=\sum_{n\in\Z^2_*}(1+|n|^2)^\sigma|a_n|^2$.
	}        
	\end{itemize}

    Let us mention that the invariance problem for $\mu_a$ essentially relies on the geometry of the quadratic Fourier interactions in the Euler equation. Following \cite{EHS-17}, we therefore introduce the following notion. 
    
	\begin{definition}\label{def 2}
	Two lattice points $p,q\in\Z^2_*$ are called a {\it degenerate pair} if
	\begin{equation*}
	(q\cdot p^\perp)\left(\frac1{|p|^2}-\frac1{|q|^2}\right)=0.
	\end{equation*}
	That is, $p,q$ are collinear or have equal length. Otherwise, $(p,q)$ is called a {\it non-degenerate pair}. 
	\end{definition}

	Our main classification result for invariant Gaussian measures is the following.

	\begin{maintheorem}\label{thm 1}
	Under assumptions  $(\mathbf{A_1})$ and $(\mathbf{A_2})$, for equation \eqref{eq euler},\eqref{eq Omega0}, let $\sigma>0$ and $a=\{a_n\}_{n\in\Z^2_*}\in h^\sigma$. Then the following statements are equivalent:
	\begin{enumerate}[label=\textup{(\roman*)},leftmargin=2.8em]
	\item $\mu_a$ is invariant under the Euler flow;
	\item $\mathcal S(a)$ contains no non-degenerate pair;
	\item $\mathcal S(a)$ is contained in a line through the origin or in a circle centered at the origin;
	\item For $\mu_a$-almost every $\Omega_{in}$, $\Phi_t(\Omega_{in})=\Omega_{in}$ for any $t\in\R$.
	\end{enumerate}
	\end{maintheorem}

    Theorem~\ref{thm 1} reveals a strong rigidity: invariance is completely determined by the geometry of the active Fourier support. The two admissible configurations correspond to shear and cellular flows, so every invariant Gaussian measure in this class is supported on steady Euler solutions; see Example~\ref{ex 1} below. The structural nature of the argument suggests that the same strategy may extend to the 3D Euler equation, where the richer Fourier interaction geometry may provide a possible route toward the classification of Gaussian invariant measures in that setting.
    \vspace{0.6em}

    Beyond invariance, quasi-invariance is also important in PDE dynamics, as it preserves sets of measure zero and hence allows almost-sure properties to be propagated by the flow. In recent years, quasi-invariance has attracted considerable attention for nonlinear dispersive equations; see, e.g., \cite{FS-22,OST-18,PTV-20}. Related equivalence phenomena have also appeared in stochastic fluid models; see \cite{CHT-25}. For the 2D Euler flow considered in this paper, we obtain the following result.

	\begin{maintheorem}\label{thm 2} Under assumptions  $(\mathbf{A_1})$ and $(\mathbf{A_2})$, for equation \eqref{eq euler},\eqref{eq Omega0}, let $\sigma>0$ and  $a=\{a_n\}_{n\in\Z^2_*}\in h^\sigma$.  Assume that $\mathscr D(a):=\{\mathbb R m:m\in\mathcal S(a)\}$ is finite. Then the following statements are equivalent:
	\begin{enumerate}[label=\textup{(\roman*)},leftmargin=2.8em]
	\item $\mu_a$ is quasi-invariant under the Euler flow;
	\item $\mu_a$ is invariant under the Euler flow.
	\end{enumerate}
	\end{maintheorem}

    Combined with the result above, this shows that, whenever $\mathcal S(a)$ is contained in a finite union of lines through the origin, quasi-invariance is equivalent to invariance and therefore yields no additional Gaussian measures beyond those supported on shear or cellular steady flows.

    \begin{remark}
    A special case of Theorem~\ref{thm 2} is when $\mathcal S(a)$ is finite. In this case, $a\in h^\sigma$ for every $\sigma>0$, so the classification in Theorem~\ref{thm 1} applies at any Sobolev regularity. Notice that the finiteness of $\mathscr D(a)$ is strictly weaker than that of $\mathcal S(a)$, since it allows infinitely many active modes along each of finitely many directions.
    \end{remark}

    \begin{example}\label{ex 1}
    The two classes of invariant Gaussian measures can be illustrated as follows.
    \begin{enumerate}[leftmargin=2.8em]
        \item[1)] Let $m\in\Z^2_*$ and $\{b_j\}_{j\in\Z\setminus\{0\}}\subset\R$ satisfy $b_{-j}=b_j$ and $\sum_{j\neq0}(1+j^2|m|^2)^\sigma|b_j|^2<\infty$. Setting $a_{jm}=b_j$ and $a_n=0$ otherwise, one has $\mathcal S(a)\subset \R m$. Hence $\mu_a$ is supported on shear-flow steady solutions and is invariant under the Euler flow.
        \item[2)]   Let $R>0$ satisfy $\{n\in\Z^2_*:|n|=R\}\neq\varnothing$, and assume that $a_n=0$ whenever $|n|\neq R$. Then  $(-\Delta)^{-1}\Omega_{in}=R^{-2}\Omega_{in}$ for $\mu_a$-almost every $\Omega_{in}$. Thus $\mu_a$ is supported on cellular-flow steady solutions and is again invariant under the Euler flow.
    \end{enumerate}      
    \end{example} 

    \subsection{Ingredients of the proof}\label{Sec 1.3}

    The proof of Theorem \ref{thm 1} is based on two key observations concerning the Fourier structure of invariant Gaussian measures. Proposition~\ref{prop mu} gives a rigidity relation for the variances along the interaction set generated by a non-degenerate pair, while Proposition~\ref{prop clo} shows that this generated set remains inside the active Fourier support.

    \begin{itemize}[leftmargin=2.8em]
    \item[1)]   The first ingredient is the rigidity property in Proposition~\ref{prop mu}. If $\mu_a$ is invariant and $p,q\in\mathcal S(a)$ form a non-degenerate pair, then there exist constants $c_0,c_1\in\R$ such that
    \begin{equation*}
    |a_n|^{-2}=c_0+\frac{c_1}{|n|^2}\quad \forall\,n\in\mathcal C(p,q),
    \end{equation*}
    where $\mathcal C(p,q)$ denotes the set of Fourier modes generated from $\{\pm p,\pm q\}$ through successive non-degenerate Euler interactions.
    
    Thus, once the inverse variances are fixed at the initial pair, the Euler interactions force the same affine relation throughout the generated set. This follows from the compatibility relation on every active triple $p+q+r=0$, $q\cdot p^\perp\neq0$, which implies that
    \begin{equation*}(|p|^{-2},|a_p|^{-2}),\quad(|q|^{-2},|a_q|^{-2}),\quad(|r|^{-2},|a_r|^{-2})
    \end{equation*}
    are collinear, and the relation can then be propagated along successive interactions.

    In particular, this affine law is precisely the covariance relation associated with the classical energy--enstrophy Gibbs family; see Remark \ref{rmk gibbs}.

    \vspace{0.6em}

    \item[2)] The second ingredient is the closure property in Proposition~\ref{prop clo}. Specifically, we show that if $p,q\in\mathcal S(a)$ form a non-degenerate pair and $\mu_a$ is invariant, then 
    \begin{equation*}
    \mathcal C(p,q)\subset\mathcal S(a).
    \end{equation*}
    
    Indeed, a non-degenerate interaction of $p,q$ necessarily creates the mode $p+q$. If $p+q$ were absent from the active support, it would remain zero almost surely under invariance, whereas its time derivative at $t=0$ contains a non-vanishing quadratic contribution from $p$ and $q$. This contradiction gives $p+q\in\mathcal S(a)$, and iteration implies the stated inclusion.
    \end{itemize}

    Combining these two ingredients, any non-degenerate active pair generates an unbounded set $\mathcal C(p,q)\subset\mathcal S(a)$ on which $|a_n|^{-2}=c_0+c_1|n|^{-2}$. This is incompatible with $a\in\ell^2(\Z^2)$ and therefore yields Theorem~\ref{thm 1}.  

    \begin{remark}
    The probabilistic and Fourier-geometric parts of the argument only require $a\in\ell^2(\Z^2)$. Consequently, the classification extends to $\sigma=0$ provided that one has a measurable solution flow on a full $\mu_a$-measure set whose trajectories conserve the $L^2$ vorticity and satisfy the Fourier-mode evolution identity used above.
    \end{remark}

    \begin{remark}
    The approach in \cite{BL-26} studies the short-time variation of averaged Sobolev norms. More precisely, their non-invariance criterion is obtained from an expansion of the form
    \begin{equation*}
    \E\|\Phi_t(\Omega_{in}^\omega)\|_{H^s}^2-\E\|\Omega_{in}^\omega\|_{H^s}^2=\gamma_{s,a}t^2+O(t^3).
    \end{equation*}
    Here $\gamma_{s,a}$ is an explicit functional of the covariance coefficients. Its sign describes the direction of the second-order variation, while the condition $\gamma_{s,a}\neq0$ for some admissible $s$ is sufficient to rule out invariance for $\sigma>3$; see \cite[Theorem A]{BL-26}.

    The present proof instead starts directly from invariance and extracts algebraic restrictions on the Fourier support and the covariance, which arise from the conditional-drift identity and the finite-dimensional Liouville relation.
    
    \end{remark}

	\subsection{Review of the literature}\label{Sec 1.4}

    The classical phenomenology of two-dimensional turbulence predicts a dual cascade: kinetic energy is transferred toward large scales, while enstrophy is transported toward small scales \cite{Kraichnan-67,Batchelor-69}; see also the surveys \cite{BE-12,Tabeling-02}. Much of this theory concerns forced-dissipative Navier--Stokes dynamics in a statistically stationary regime. For randomly forced 2D Navier--Stokes equations, ergodicity and mixing of the associated Markov dynamics provide a rigorous framework for describing such long-time statistics; see e.g. \cite{FM-95,EMS-01,KS-00,HM-06} and references therein. On the Lagrangian side, recent work has established chaotic particle dynamics and the Batchelor spectrum \cite{BBPS-22a,BBPS-22c}.
    
    By contrast, random initial data for the inviscid Euler flow are more closely related to decaying turbulence, where transient enstrophy transfer and the emergence of large-scale coherent vortices are commonly observed \cite{RDCE-03,Tabeling-02}. This statistical viewpoint motivates the study of how probability laws of random vorticity are transported or preserved by the Euler dynamics. From a deterministic viewpoint, the long-time dynamics near Euler steady states have also been widely studied. Inviscid damping near shear flows and axi-symmetrization near point vortices provide two representative relaxation mechanisms; see, e.g., \cite{BM-15,IJ-20,IJ-23}.

    \vspace{0.6em}

    Invariant measures provide a rigorous framework for Euler statistics. More broadly, following Bourgain's seminal work \cite{Bourgain-94}, probabilistic Cauchy theory and invariant-measure arguments have become central tools for constructing almost-sure global dynamics in nonlinear dispersive equations; see, e.g., \cite{BT-08,BT-14,CO-12,DNY-24}. For the 2D Euler equation, explicit Gibbs-type and white-noise Gaussian invariant measures have been constructed at low regularity \cite{AC-90,Flandoli-18}.

    A different construction is provided by the fluctuation-dissipation method, which builds invariant measures for randomly forced viscous approximations and then takes an inviscid limit \cite{Kuksin-04}. Using hyperviscous approximations, Latocca \cite{Latocca-23} constructed non-atomic Euler invariant measures  concentrated on $H^s$ velocity fields for every $s>2$. As these measures are obtained through compactness, their precise structure remains largely unknown; in particular, it is not known whether they are Gaussian or supported entirely on steady solutions.

    \vspace{0.6em}

	The remainder of the paper is organized as follows. Section Section~\ref{Sec 2} collects the well-posedness results and establishes the conditional-drift reduction used below. In Section \ref{Sec 3}, we establish the geometric structure of invariant Gaussian measures through the closure of the active Fourier support and the compatibility relation on active triples. Finally, we combine these ingredients to prove the two classification theorems in Section \ref{Sec 4}.

	\section{Probabilistic setup and conditional reduction}\label{Sec 2}

	In this section, we first summarize the mathematical setup for the Euler equation with Gaussian initial data.  Additionally, we identify the conditional drift of the Euler nonlinearity with its finite-dimensional Galerkin projection.

	\vspace{0.6em}

	We use the following Fourier conventions. For $n=(n_1,n_2)\in\Z^2_*$, write
	\begin{equation*}
	n^\perp=(-n_2,n_1),\quad |n|^2=n_1^2+n_2^2.
	\end{equation*}
    
	The normalized Fourier coefficients for zero mean functions are denoted by
	\begin{equation*}
	L_0^2(\T^2)=\left\{f(x)=\sum_{n\in\Z^2_*}f_ne^{in\cdot x}\in L^2(\T^2):\int_{\T^2}fdx=0\right\},\quad f_n=\frac1{(2\pi)^2}\int_{\T^2}f(x)e^{-in\cdot x}dx,
	\end{equation*}
	where $L_0^2(\T^2)$ is endowed with the usual $L^2$-inner product 
	\begin{equation*}
	\langle f,g\rangle_{L^2}=\int_{\T^2}f(x)g(x)dx,\quad\|f\|_{L^2}^2=(2\pi)^2\sum_{n\in\Z^2_*}|f_n|^2.
	\end{equation*} 

    For $\sigma>0$, we denote by $H^\sigma(\T^2)$ the  Sobolev space
	\begin{equation*}
	H^\sigma(\T^2)=\left\{f\in L^2(\T^2):\|f\|_{H^\sigma}<\infty\right\},\quad\|f\|_{H^\sigma}^2=\sum_{n\in\Z^2}(1+|n|^2)^\sigma|f_n|^2.
	\end{equation*}
	Note that for $a=\{a_n\}_{n\in\Z_*^2}\in h^\sigma$, the random initial satisfies $\Omega_{in}^\omega\in H^\sigma(\T^2)$ almost surely. Thus $\mu_a$ is a Gaussian measure on $H^\sigma(\T^2)$.

	\subsection{The bilinear form and well-posedness}\label{Sec 2.1} Let us define the Euler bilinear form initially for smooth zero-mean functions $f,g\colon \T^2\rightarrow\R$ by 
	\begin{equation*}
	U[f]:=\nabla^\perp(-\Delta)^{-1}f,\quad B(f,g):=U[f]\cdot\nabla g,\quad B(f):=B(f,f).
	\end{equation*}

	In this notation, equation \eqref{eq euler} can be rewritten as $\partial_t\Omega=-B(\Omega)$. For each $k\in\Z^2_*$, a direct computation gives
	\begin{equation}\label{eq B}
	B(f,g)_k=-\sum_{\substack{p+q=k,\;p,q\in\Z^2_*}}\frac{q\cdot p^\perp}{|p|^2}f_pg_q.
	\end{equation}

	To describe the structure of the bilinear form, for any $p,q\in\Z^2_*$, we set
	\begin{equation}\label{eq kappa}
	\kappa(p,q):=\frac12(q\cdot p^\perp)\left(\frac1{|q|^2}-\frac1{|p|^2}\right).
	\end{equation}
	Note that $\kappa(p,q)=\kappa(q,p)$, and  $\kappa(p,q)=0$ if and only if $(p,q)$ is a degenerate pair as in Definition \ref{def 2}. Pairing the terms indexed by $(p,q)$ and $(q,p)$ in \eqref{eq B}, we derive 
	\begin{equation}\label{eq B2}
	B(f)_k=\sum_{\substack{p+q=k,\;p,q\in\Z^2_*}}\kappa(p,q)f_pf_q,\quad k\in\Z^2_*.
	\end{equation}

	Additionally, we have the following estimate.
	\begin{lemma} \label{lemma B}
	For any $k\in\Z^2_*$ and $f,g\in L_0^2(\T^2)$, it follows that
	\begin{equation*}
	|B(f,g)_k|\leq \frac{|k|}{4\pi^2}\|f\|_{L^2}\|g\|_{L^2}.
	\end{equation*}
	Consequently, the fixed Fourier coefficient $B(f,g)_k$, initially defined for smooth functions with zero mean, extends uniquely and continuously to $L_0^2(\T^2)\times L_0^2(\T^2)$.
	\end{lemma}

	\begin{proof}
	Since $U[f]$ is divergence-free,
	\begin{equation*}
	B(f,g)=\nabla\cdot\left(U[f]g\right).
	\end{equation*}
	Using the normalized Fourier coefficients fixed above and integrating by parts, we obtain
	\begin{equation*}
	\begin{aligned}
	B(f,g)_k=\frac1{(2\pi)^2}\int_{\T^2}\nabla\cdot\left(U[f]g\right)e^{-ik\cdot x}dx=\frac{i}{(2\pi)^2}\int_{\T^2}k\cdot U[f]g\,e^{-ik\cdot x}dx.
	\end{aligned}
	\end{equation*}

	Therefore, using H\"older's inequality, one has
	\begin{equation}\label{eq B-estimate}
	|B(f,g)_k|\leq\frac{|k|}{4\pi^2}\|U[f]g\|_{L^1}\leq\frac{|k|}{4\pi^2}\|U[f]\|_{L^2}\|g\|_{L^2}.
	\end{equation}
	Since $f$ has zero mean, Parseval's identity yields
	\begin{equation*}
	\|U[f]\|_{L^2}^2=(2\pi)^2\sum_{n\in\Z^2_*}\frac{|f_n|^2}{|n|^2}\leq(2\pi)^2\sum_{n\in\Z^2_*}|f_n|^2=\|f\|_{L^2}^2.
	\end{equation*}
	Combining this estimate with \eqref{eq B-estimate}, we complete the proof of Lemma \ref{lemma B}.
	\end{proof}

    The following proposition summarizes the basic well-posedness and regularity for equation \eqref{eq euler} with Gaussian initial data \eqref{eq Omega0}.

	\begin{proposition}\label{prop wp} Under assumptions  $(\mathbf{A_1})$ and $(\mathbf{A_2})$, for equation \eqref{eq euler},\eqref{eq Omega0}, let $\sigma>0$ and $a=\{a_n\}_{n\in\Z^2_*}\in h^\sigma$. Then  for any $\alpha\in(0,\sigma\wedge1)$, 
	\begin{equation*}
	\Omega_{in}^\omega\in c_0^\alpha(\T^2):=\overline{C^\infty(\mathbb{T}^2)\cap L^2_0(\mathbb{T}^2)}^{C^\alpha(\mathbb{T}^2)}
	\quad \text{almost surely}.
	\end{equation*}
    Moreover, for any $\Omega_{in}\in c_0^\alpha(\T^2)$, equation \eqref{eq euler} admits a unique global solution $\Omega\in C(\R;c_0^\alpha(\T^2))$, which defines a Borel flow
	\begin{equation*}
	\Phi_t: c_0^\alpha(\T^2)\rightarrow c_0^\alpha(\T^2),\quad\Phi_t(\Omega_{in}):=\Omega(t),\quad t\in\R.
	\end{equation*}
    For each $k\in\mathbb{Z}^2_*$, one has $\Omega_k\in C^1(\mathbb{R})$ and
	\begin{equation*}
	\frac{d}{dt}\Omega_k(t)=-B(\Omega(t))_k.
	\end{equation*}
	\end{proposition}

    \begin{proof}
	By the regularity theory of Gaussian Fourier series \cite{Kahane-93}, for any $p\geq2$, one has
	\begin{equation*}
	\Omega_{in}^\omega\in W^{\sigma,p}(\T^2)\quad\text{almost surely},\quad\E\left\|\Omega_{in}^\omega\right\|_{W^{\sigma,p}}^2<\infty.
	\end{equation*}
	Fix $\alpha\in(0,\sigma\wedge1)$ and choose $\beta\in(\alpha,\sigma\wedge1)$. Thus by taking $p<\infty$ sufficiently large such that $\beta<\sigma-2/p$, the Sobolev--Morrey embedding  $W^{\sigma,p}(\T^2)\hookrightarrow C^\beta(\T^2)\hookrightarrow c^\alpha(\T^2)$ then ensures that $\Omega_{in}^\omega\in c_0^\alpha(\T^2)$ almost surely. 
    
    The remaining assertions are deterministic and apply pathwise. The periodic Biot--Savart operator maps $c_0^\alpha(\T^2)$ continuously into the divergence-free subspace of $c^{1,\alpha}(\T^2)$. Hence the local Hadamard well-posedness argument of \cite[Theorem~1.2]{MY-18}, stated on $\R^2$, carries over to $\T^2$: the periodic kernel has the same local singularity as the Euclidean one, and the required Schauder and Lagrangian estimates remain unchanged. In two dimensions, conservation of the vorticity $L^\infty$-norm and the standard continuation criterion extend the solution globally; see \cite[Remark~1.3]{MY-18} and also \cite{MB-02,Wolibner-33}. The continuous dependence on the initial datum then makes each $\Phi_t$ continuous, and hence Borel measurable.
    
    Finally, noting that $\Omega\in C(\mathbb{R};c^\alpha_0(\mathbb{T}^2))\subset C(\mathbb{R};L^2_0(\mathbb{T}^2))$ and using Lemma \ref{lemma B}, one has
	\begin{equation*}
	\left|B(\Omega(t))_k-B(\Omega(s))_k\right|
	\leq\frac{|k|}{4\pi^2}\left(\|\Omega(t)\|_{L^2}+\|\Omega(s)\|_{L^2}\right)\|\Omega(t)-\Omega(s)\|_{L^2},
	\end{equation*}
	and thus $t\mapsto B(\Omega(t))_k$ is continuous. Meanwhile, by \eqref{eq euler}, we obtain
	\begin{equation*}
	\Omega_k(t)=\Omega_k(0)-\int_0^tB(\Omega(s))_kds,\quad t\in\mathbb{R},
	\end{equation*}
	which thus ensures $\Omega_k\in C^1(\mathbb{R})$ and the needed relation. This completes the proof.
    \end{proof}

	\subsection{Conditional drift and Galerkin projection}\label{Sec 2.2} 
    
    We now pass from the full Euler drift to its finite-dimensional conditional version. This reduction is the point at which the independence of the Gaussian Fourier coefficients enters the proof.

    \vspace{0.6em}
    Let $\sigma>0$ and  $a=\{a_n\}_{n\in\Z^2_*}\in h^\sigma$. By definition, the active Fourier support $\mathcal S(a)$ is symmetric. Let $\Lambda=-\Lambda\subset\mathcal S(a)$ be a finite and symmetric set in $\Z^2_*$ and set
	\begin{equation*}
	\Lambda_+=\Lambda\cap\Z^2_+,\quad\mathcal F_\Lambda=\sigma(g_j^\omega:j\in\Lambda_+).
	\end{equation*}
	For $f\in L^2(\T^2)$, define
	\begin{equation*}
	P_\Lambda f(x):=\sum_{j\in\Lambda}f_je^{ij\cdot x}.
	\end{equation*}
	For $k\in\Z^2_*$, using \eqref{eq kappa},\eqref{eq B2}, we define the finite Galerkin sum by
	\begin{equation}\label{eq B-lambda}
	B_k^\Lambda(f):=B(P_\Lambda f)_k
	=\sum_{\substack{p+q=k,\;p,q\in\Lambda}}\kappa(p,q)f_pf_q.
	\end{equation}

	\begin{lemma}\label{lemma B2}
	For any $k\in\Z^2_*$ and finite symmetric $\Lambda\subset\mathcal S(a)$, it follows that
	\begin{equation*}
	\E\left[B(\Omega_{in}^\omega)_k|\mathcal F_\Lambda\right]=B_k^\Lambda(\Omega_{in}^\omega)\quad\text{in }L^2(\Xi).
	\end{equation*}
	\end{lemma}

	\begin{proof}
	For $N\geq1$, set
	\begin{equation*}
	\Omega_{in,N}^\omega(x):=\sum_{\substack{n\in\Z^2_*,\;|n|\leq N}}a_ng_n^\omega e^{in\cdot x}.
	\end{equation*}
	Recall that $a\in h^\sigma\subset\ell^2(\Z^2_*)$. Then using the independence of the Gaussian variables, one derives 	\begin{equation}\label{eq OmegaN}
	\lim\limits_{N\rightarrow\infty}\E\|\Omega_{in,N}^\omega-\Omega_{in}^\omega\|_{L^2}^4=0,\quad \sup_{N\geq1}\E\|\Omega_{in,N}^\omega\|_{L^2}^4\leq C(\E\|\Omega_{in}^\omega\|_{L^2}^2)^2<\infty.
	\end{equation}

	Meanwhile, by the bilinearity of $B$ and Lemma~\ref{lemma B}, we have 
	\begin{align*}
	|B(\Omega_{in,N}^\omega)_k-B(\Omega_{in}^\omega)_k|&\leq |B(\Omega_{in,N}^\omega-\Omega_{in}^\omega,\Omega_{in,N}^\omega)_k|+|B(\Omega_{in}^\omega,\Omega_{in,N}^\omega-\Omega_{in}^\omega)_k|\\
	&\leq \frac{|k|}{4\pi^2}\left(\|\Omega_{in,N}^\omega\|_{L^2}+\|\Omega_{in}^\omega\|_{L^2}\right)\|\Omega_{in,N}^\omega-\Omega_{in}^\omega\|_{L^2}.
	\end{align*}    
	Therefore, applying \eqref{eq OmegaN}, we obtain
    \begin{equation*}
        \E|B(\Omega_{in,N}^\omega)_k-B(\Omega_{in}^\omega)_k|^2\leq C_k(\E\|\Omega_{in}^\omega\|_{L^2}^4)^{1/2}(\E\|\Omega_{in,N}^\omega-\Omega_{in}^\omega\|_{L^2}^4)^{1/2}
    \end{equation*}
    which implies that 
	\begin{equation}\label{eq OmegaN2}
	B(\Omega_{in,N}^\omega)_k\longrightarrow B(\Omega_{in}^\omega)_k
	\quad\text{in }L^2(\Xi).
	\end{equation}

	Let us then fix $N$ large enough such that $\Lambda\subset\{|n|\leq N\}$. The $p$-th Fourier mode of $\Omega_{in,N}^\omega$ is $a_pg_p^{\omega}e^{ip\cdot x}$, whose Fourier coefficient is $a_pg_p^\omega$. By the Fourier representation \eqref{eq B2}, we have
	\begin{equation}\label{eq BN}
	B(\Omega_{in,N}^\omega)_k=\sum_{\substack{p+q=k,\;p,q\in\Z^2_*,\;|p|,|q|\leq N}}\kappa(p,q)a_pa_qg_p^{\omega}g_q^{\omega}.
	\end{equation}
	For any $p,q\in\Z^2_*$ with $|p|,|q|\leq N$, by the Gaussian structure and the symmetry of $\Lambda$, we derive
	\begin{equation}\label{eq Gaussian}
	\E\left[g_p^{\omega}g_q^{\omega}|\mathcal F_\Lambda\right]=\begin{cases}
	g_p^{\omega}g_q^{\omega},&p,q\in\Lambda,\\
	0,&p\in\Lambda,q\notin\Lambda\text{ or }p\notin\Lambda,q\in\Lambda,\\
	2\mathbf{1}_{\{q=-p\}},&p,q\notin\Lambda.
	\end{cases}
	\end{equation}
	Indeed, the first product in \eqref{eq Gaussian} is $\mathcal F_\Lambda$-measurable. In the second case, the factor outside $\Lambda$ is centered and independent of $\mathcal F_\Lambda$ and of the other factor. In the last case, both factors are independent of $\mathcal F_\Lambda$, and $\E\left[g_p^{\omega}g_q^{\omega}\right]=2\mathbf{1}_{\{q=-p\}}$.
    
	Since $p+q=k\neq0$, the relation $q=-p$ cannot occur in the sum \eqref{eq BN}. Therefore, applying \eqref{eq Gaussian} term by term in \eqref{eq BN}, one obtains
	\begin{align*}
	\E\left[B(\Omega_{in,N}^\omega)_k|\mathcal F_\Lambda\right]=\sum_{\substack{p+q=k,\;p,q\in\Lambda}}\kappa(p,q)a_pa_qg_p^{\omega}g_q^{\omega}=B_k^\Lambda(\Omega_{in}^\omega).
	\end{align*}
	Finally, noticing that conditional expectation is an $L^2(\Xi)$ contraction and using \eqref{eq OmegaN2}, we obtain
	\begin{align*}
	\E\left|\E\left[B(\Omega_{in,N}^\omega)_k-B(\Omega_{in}^\omega)_k|\mathcal F_\Lambda\right]\right|^2
	&\leq \E\left|B(\Omega_{in,N}^\omega)_k-B(\Omega_{in}^\omega)_k\right|^2\longrightarrow0\quad\text{as }N\rightarrow\infty.
	\end{align*}

	Combining the last two displays, we conclude that
	\begin{equation*}
	\E\left[B(\Omega_{in}^\omega)_k|\mathcal F_\Lambda\right]=B_k^\Lambda(\Omega_{in}^\omega)\quad\text{in }L^2(\Xi).
	\end{equation*}
	This completes the proof.
	\end{proof}

	\section{Rigidity properties of invariant measures}\label{Sec 3}

    This section examines the restrictions imposed by invariance on the active Fourier support and the associated variances. Proposition~\ref{prop clo} shows that the support is closed under every non-degenerate Euler interaction, while Proposition~\ref{prop mu} propagates an affine law for the inverse variances throughout the generated set. Together, these two properties provide the rigidity needed for the classification in Section~\ref{Sec 4}.
    
	\vspace{0.6em}

	Fix a non-degenerate pair $p,q\in\Z^2_*$. We define recursively
	\begin{align*}
	\mathcal C_0(p,q)&=\{\pm p,\pm q\},\\
	\mathcal C_{j+1}(p,q)&=\mathcal C_j(p,q)\cup\left\{\pm(m+n):m,n\in\mathcal C_j(p,q),(m,n)\text{ is non-degenerate}\right\},
	\end{align*}
	and set
	\begin{equation*}
	\mathcal C(p,q)=\bigcup_{j\in\N}\mathcal C_j(p,q).
	\end{equation*}
	Note that a non-degenerate pair cannot sum to zero, thus $0\notin\mathcal{C}(p,q)$. Additionally, the generated set is infinite. 
    \begin{lemma}\label{lemma C}
        Let $p,q\in\Z_*^2$ form a non-degenerate pair. Then the set $\mathcal C(p,q)$ is infinite.
    \end{lemma}

    \begin{remark}
    Let us note that the construction of $\mathcal C(p,q)$ follows the same lattice-generation mechanism as the interaction closure in \cite[Section~4.1]{HM-06}. The viscous term considered there is diagonal in Fourier space and therefore does not affect this geometry. For completeness, we give a direct proof in Section \ref{Sec 3.3}. 
    \end{remark}

    The following proposition is the main structural rigidity result used to establish the classification of invariant measures, whose proof is placed at the end of this section.

	\begin{proposition}\label{prop mu}
	Under assumptions  $(\mathbf{A_1})$ and $(\mathbf{A_2})$, for equation \eqref{eq euler},\eqref{eq Omega0}, let $\sigma>0$ and $a=\{a_n\}_{n\in\Z^2_*}\in h^\sigma$. Assume that $\mu_a$ is invariant and that $p,q\in\mathcal S(a)$ form a non-degenerate pair. Then there exist constants $c_0,c_1\in\R$ such that
	\begin{equation}\label{eq mu}
	|a_n|^{-2}=c_0+\frac{c_1}{|n|^2}\quad \forall\,n\in\mathcal C(p,q).
	\end{equation}
	\end{proposition}

    \begin{remark}\label{rmk gibbs}
    The affine relation \eqref{eq mu} has a natural Gibbs interpretation. Indeed, the enstrophy and kinetic energy of a vorticity field $\Omega$ are formally given, up to normalization constants, by
    \begin{equation*}
    \mathcal Z(\Omega)
    =\frac12\sum_{n\in\mathbb Z^2_*}|\Omega_n|^2,\quad\mathcal E(\Omega)=\frac12\sum_{n\in\mathbb Z^2_*}\frac{|\Omega_n|^2}{|n|^2}.
    \end{equation*}
    
    Consequently, the formal energy--enstrophy Gibbs measure
    \begin{equation*}
    d\mu_{\alpha,\beta}(\Omega)=Z_{\alpha,\beta}^{-1}\exp\bigl(-\alpha\mathcal Z(\Omega)-\beta\mathcal E(\Omega)\bigr)\,d\Omega
    \end{equation*}
    is a diagonal Gaussian measure whose Fourier variances satisfy, after relabeling the constants,
    \begin{equation*}
    |a_n|^{-2}=c_0+\frac{c_1}{|n|^2}.
    \end{equation*}
    Thus, Proposition~\ref{prop mu} may be viewed as an interaction-generated Gibbs rigidity statement: the presence of one non-degenerate active pair forces the covariance to agree with the energy--enstrophy Gibbs law throughout $\mathcal C(p,q)$, although not necessarily on the entire Fourier lattice.

    The classical full-support Gibbs measures are too rough for the positive-regularity setting considered here. In the usual case $c_0>0$, one has $|a_n|^2\to c_0^{-1}$ as $|n|\to\infty$, so the corresponding vorticity has the regularity of spatial white noise and belongs only to $H^s$ for $s<-1$. If $c_0=0$, the variances grow like $|n|^2$, while $c_0<0$ is incompatible with positivity at sufficiently high frequencies. This incompatibility between the Gibbs-type covariance forced by Euler interactions and the assumption $a\in h^\sigma$, $\sigma>0$, is the basic rigidity mechanism underlying the classification.
    \end{remark}

	\subsection{Closure of the active Fourier support}\label{Sec 3.1}

	Let us first show that (quasi-)invariance forces the active Fourier support to be closed under every non-degenerate interaction.

	\begin{proposition}\label{prop clo}
	Under assumptions  $(\mathbf{A_1})$ and $(\mathbf{A_2})$, for equation \eqref{eq euler},\eqref{eq Omega0}, let $\sigma>0$, $a=\{a_n\}_{n\in\Z^2_*}\in h^\sigma$, and $p,q\in\mathcal S(a)$ form a non-degenerate pair. Assume that either $\mu_a$ is invariant or there exists a sequence $t_j\rightarrow0$, $t_j\neq0$, such that
	\begin{equation}\label{eq clo}
	(\Phi_{t_j})_*\mu_a\ll\mu_a\quad\forall\,j\in\N^+.
	\end{equation}
	Then $\mathcal{C}(p,q)\subset\mathcal S(a)$. 
	\end{proposition}

	\begin{proof}
	If $\mu_a$ is invariant, then \eqref{eq clo} holds for any sequence of nonzero times converging to zero. Thus, in either case, we may fix a sequence $\{t_j\}_{j\in\N^+}$ satisfying \eqref{eq clo}. 
    
    Clearly, $\mathcal{C}_0(p,q)=\{\pm p,\pm q\}\subset \mathcal{S}(a)$.  We first show that $\mathcal{C}_1(p,q)=\mathcal{C}_0(p,q)\cup \{\pm (p\pm q)\}\subset \mathcal{S}(a)$.  By symmetry, it suffices to prove that $k=p+q$ belongs to $\mathcal S(a)$. Assume for contradiction that $k\notin\mathcal S(a)$ and set
	\begin{equation*}
	A_k=\left\{w\in L_0^2(\T^2):w_k=0\right\}.
	\end{equation*}
	Note that by assumption,  $\mu_a(A_k^c)=0$. Additionally, invoking \eqref{eq clo}, one has
	\begin{equation*}
	(\Phi_{t_j})_*\mu_a(A_k^c)=0\quad\forall\,j\in\N^+,
	\end{equation*}
	and hence
	\begin{equation*}
	\mu_a\left(\Phi_{t_j}^{-1}(A_k^c)\right)=0\quad\forall\,j\in\N^+.
	\end{equation*}    
	Therefore, as the extension $\Phi_{t_j}$ is Borel measurable, the set $E:=c_0^\alpha(\T^2)\cap A_k\cap\bigcap_{j\in\N^+}\Phi_{t_j}^{-1}(A_k)$ is measurable and satisfies $\mu_a(E)=1$. In particular, for equation \eqref{eq euler} with initial datum  $\Omega_{in}\in E$, it follows that
	\begin{equation*}
	\Omega_k(0)=0,\quad\Omega_k(t_j)=0\quad\forall\,j\in\N^+.
	\end{equation*}

	Meanwhile, note that $\Omega_k(\cdot)$ is continuously differentiable with
	\begin{equation*}
	\frac d{dt}\Omega_k(t)=-B(\Omega(t))_k.
	\end{equation*}
	Therefore,  combining these two relations above, we derive that
	\begin{equation*}	\partial_t\Omega_k(0)=\lim\limits_{j\rightarrow\infty}\frac{\Omega_k(t_j)-\Omega_k(0)}{t_j}=0=-B(\Omega_{in}^\omega)_k\quad\Pb\text{-almost surely}.
	\end{equation*}

	We now apply Lemma \ref{lemma B2} with  $\Lambda:=\{\pm p,\pm q\}$, which implies 
	\begin{equation}\label{eq clo1}
	0=\E\left[B(\Omega_{in}^\omega)_k|\mathcal F_\Lambda\right]=B_k^\Lambda(\Omega_{in}^\omega)\quad\Pb\text{-almost surely}.
	\end{equation}    
	On the other hand, we directly compute that 
	\begin{equation}\label{eq clo2}
	B_k^\Lambda(\Omega_{in}^\omega)=\sum_{\substack{p+q=k,\;p,q\in\Lambda}}\kappa(p,q)a_pa_qg_p^{\omega}g_q^{\omega}=(q\cdot p^\perp)\left(\frac1{|q|^2}-\frac1{|p|^2}\right)a_pa_qg_p^{\omega}g_q^{\omega}.
	\end{equation}

	Therefore, combining the non-degeneracy, $p,q\in\mathcal S(a)$ and \eqref{eq clo1},\eqref{eq clo2}, we have
	\begin{equation*}
	g_p^{\omega}g_q^{\omega}=0\quad\Pb\text{-almost surely},
	\end{equation*}
	which is a contradiction. This implies $\mathcal C_1(p,q)\subset\mathcal S(a)$.

    Consequently, we conclude that $\mathcal C(p,q)\subset\mathcal S(a)$ holds by induction.  This completes the proof of Proposition~\ref{prop clo}.
	\end{proof}

	\subsection{Affine compatibility on active triples}\label{Sec 3.2}

	We next turn from the geometry of the support to the covariance structure of the measure.

	\begin{proposition}\label{prop triple} 
	Under assumptions  $(\mathbf{A_1})$ and $(\mathbf{A_2})$, for equation \eqref{eq euler},\eqref{eq Omega0}, let $\sigma>0$ and $a=\{a_n\}_{n\in\Z^2_*}\in h^\sigma$. Assume that $\mu_a$ is invariant and that $p,q,r\in\mathcal S(a)$ form an active triple in the sense that
	\begin{equation*}
	p+q+r=0,\quad q\cdot p^\perp\neq0.
	\end{equation*}
	Then
	\begin{equation}\label{eq triple}
	\frac{1}{|a_r|^2}\left(\frac1{|p|^2}-\frac1{|q|^2}\right)+\frac{1}{|a_p|^2}\left(\frac1{|q|^2}-\frac1{|r|^2}\right)+\frac{1}{|a_q|^2}\left(\frac1{|r|^2}-\frac1{|p|^2}\right)=0,
	\end{equation} 
	Equivalently, the three points $(|p|^{-2},|a_p|^{-2})$, $(|q|^{-2},|a_q|^{-2})$, $(|r|^{-2},|a_r|^{-2})$ are collinear in $\R^2$.
	\end{proposition}

	This result essentially relies on the Gaussian structure of $\mu_a$. Let us begin with the following lemma. Let $\Lambda=-\Lambda\subset\mathcal S(a)$ be a finite symmetric set and $\Lambda_+=\Lambda\cap\Z^2_+$. For $z=\{z_k\}_{k\in\Lambda_+}$ on $\C^{|\Lambda_+|}$, extend the coordinates by $z_{-j}=\overline{z_j}$ and define 
	\begin{equation*}
	\V^\Lambda(z)=(\V_k^\Lambda(z))_{k\in\Lambda_+},\quad \V_k^\Lambda(z)=-\sum_{\substack{p+q=k,\;p,q\in\Lambda}}\kappa(p,q)z_pz_q,\quad k\in\Lambda_+,\;z\in\C^{|\Lambda_+|}.
	\end{equation*}
	We also view $\V^\Lambda$ as a real vector field on $\R^{2|\Lambda_+|}$. The $\Lambda$-marginal of $\mu_a$ has density
	\begin{equation*}
	\rho_\Lambda(z)=\prod_{k\in\Lambda_+}\frac1{2\pi|a_k|^2}\exp\left(-\frac{|z_k|^2}{2|a_k|^2}\right)=:C_\Lambda e^{-Q_\Lambda(z)},\quad Q_\Lambda(z)=\sum_{k\in\Lambda_+}\frac{|z_k|^2}{2|a_k|^2}.
	\end{equation*}

	\begin{lemma}\label{lemma V}
	Under assumptions $\mathbf{(A_1)}$ and $\mathbf{(A_2)}$, for equation \eqref{eq euler},\eqref{eq Omega0}, let $\sigma>0$ and $a=\{a_n\}_{n\in\Z^2_*}\in h^\sigma$. Assume that $\mu_a$ is invariant. Then, for any finite symmetric set $\Lambda\subset\mathcal S(a)$, it follows that 
	\begin{equation}\label{eq V}
	\V^\Lambda(z)\cdot\nabla Q_\Lambda(z)=0\quad\forall\,z\in\C^{|\Lambda_+|}.
	\end{equation}
	\end{lemma}

    \begin{remark}
    Lemma~\ref{lemma V} admits a finite-dimensional Liouville interpretation. Indeed, invariance of $\mu_a$ and conditioning on the $\Lambda$-modes imply that the marginal density $\rho_\Lambda$ is stationary under the conditional drift $\V^\Lambda$:
    \begin{equation*}
    \dvg\left(\rho_\Lambda\V^\Lambda\right)=0.
    \end{equation*}
    As $\V^\Lambda$ is divergence-free and $\rho_\Lambda=C_\Lambda e^{-Q_\Lambda}$, we therefore derive \eqref{eq V}. Thus, $\V^\Lambda$ is tangent to the level sets of $Q_\Lambda$, and $Q_\Lambda$ is a first integral of the associated finite-dimensional system $\dot z=\V^\Lambda(z)$. 
    
    This reverses the usual Gibbs-measure principle: normally, a divergence-free dynamics together with conservation of $Q_\Lambda$ implies invariance of the density $e^{-Q_\Lambda}$, whereas here invariance of the Gaussian density forces its quadratic exponent to be conserved.
    \end{remark}

	Invoking this lemma, we then establish Proposition \ref{prop triple} as follows.

	\begin{proof}[Proof of Proposition \ref{prop triple}]
	Set
	\begin{equation*}
	\Lambda=\{\pm p,\pm q,\pm r\}\subset\mathcal S(a),
	\quad
	\Lambda_+=\Lambda\cap\Z^2_+,
	\end{equation*}
	Notice that
	$q\cdot p^\perp\neq0$ and $p+q+r=0$ imply that $p,q,r$ belong
	to three distinct sign classes. We do not, however, assume that
	$p,q,r$ themselves are the chosen representatives in $\Lambda_+$.

	For every $k\in\Lambda\setminus\Lambda_+$, extend the definition of $\V_k^\Lambda$ by
	\begin{equation*}
	\V_k^\Lambda(z)
	:=-\sum_{\substack{m+n=k,\;m,n\in\Lambda}}
	\kappa(m,n)z_mz_n,
	\quad z_{-\ell}=\overline{z_\ell},\quad\ell\in\Lambda_+.
	\end{equation*}
	Then $\V_{-k}^\Lambda=\overline{\V_k^\Lambda}$ and $|a_{-k}|^{-2}=|a_k|^{-2}$.
	In what follows, we shall work entirely in real coordinates. Specifically, for each
	$k\in\Lambda_+$, write
	\begin{equation*}
	z_{\pm k}=x_k\pm iy_k,\quad\V_k^\Lambda(z)=U_k^\Lambda(x,y)+iW_k^\Lambda(x,y),\quad Q_\Lambda(z)=\frac12\sum_{k\in\Lambda_+}\frac{1}{|a_k|^2}(x_k^2+y_k^2).
	\end{equation*}	 
    
	In particular, by Lemma \ref{lemma V}, it follows that for any $x,y\in\R^{|\Lambda_+|}$,
	\begin{align}\label{eq t2}
	P(x,y):=\V^\Lambda(x,y)\cdot\nabla Q_\Lambda(x,y)=\sum_{k\in\Lambda_+}\frac{1}{|a_k|^2}\left(x_kU_k^\Lambda(x,y)+y_kW_k^\Lambda(x,y)\right)=0.	
	\end{align} 

	Let $\widehat p,\widehat q,\widehat r\in\Lambda_+$ denote the
	chosen representatives of the three sign classes
	$\{\pm p\},\{\pm q\},\{\pm r\}$, respectively. Note that they are distinct
	because $q\cdot p^\perp\neq0$. Set
	\begin{equation*}
	x_{\widehat p}=\xi_p,\quad x_{\widehat q}=\xi_q,\quad x_{\widehat r}=\xi_r.
	\end{equation*}
	In view of \eqref{eq t2}, we may restrict $P$ to the real three-dimensional coordinate slice by letting $y=0$. On this slice, one has
	\begin{equation*}
	z_p=z_{-p}=\xi_p,\quad z_q=z_{-q}=\xi_q,\quad z_r=z_{-r}=\xi_r.
	\end{equation*}
	We then denote the corresponding cubic polynomial by $\widetilde P(\xi_p,\xi_q,\xi_r)=P(x,0)$. Consequently, applying \eqref{eq t2} again, we compute that
	\begin{equation*}
	\frac{\partial^3\widetilde P}{\partial\xi_p\,\partial\xi_q\,\partial\xi_r}(0,0,0)=0,
	\end{equation*}
	which implies the
	coefficient of $\xi_p\xi_q\xi_r$ in $\widetilde P$ is $0$.

	On the other hand, let us compute this coefficient directly. Using the elementary real identity
	\begin{equation*}
	2(x_kU_k^\Lambda+y_kW_k^\Lambda)
	=z_{-k}\V_k^\Lambda+z_k\V_{-k}^\Lambda
	\end{equation*}
	and the symmetries $|a_{-k}|^{-2}=|a_k|^{-2}$, 
	$\V_{-k}^\Lambda=\overline{\V_k^\Lambda}$, we derive
	\begin{equation}\label{eq t3}
	P(x,y)=\frac12\sum_{k\in\Lambda}
	\frac{z_{-k}}{|a_k|^2}\V_k^\Lambda(z).
	\end{equation} 

    Moreover, using $p+q+r=0$ and setting $y=0$, we directly compute that
    \begin{equation*}
        \V_r^\Lambda=-2\kappa(-p,-q)z_{-p}z_{-q}=-2\kappa(p,q)z_pz_q,\quad\V_{-r}^\Lambda=-2\kappa(p,q)z_pz_q,
    \end{equation*}
    Similarly, one has
    \begin{equation*}
        \V_{\pm p}^\Lambda(z)=-2\kappa(q,r)\xi_q\xi_r,\quad \V_{\pm q}^\Lambda(z)=-2\kappa(r,p)\xi_r\xi_p.
    \end{equation*}    
    Taking $y=0$ and plugging these relations into \eqref{eq t3}, the coefficient of $\xi_p\xi_q\xi_r$ in $\widetilde P$ equals
    \begin{equation}\label{eq t4}
        -2\left(|a_r|^{-2}\kappa(p,q)+|a_p|^{-2}\kappa(q,r)+|a_q|^{-2}\kappa(r,p)\right).
    \end{equation}

    Collecting \eqref{eq kappa},\eqref{eq t3},\eqref{eq t4} with the relation $r\cdot q^\perp=p\cdot r^\perp=q\cdot p^\perp$, we thus conclude that 
	\begin{align*}
	(q\cdot p^\perp)\left(\frac{1}{|a_r|^2}\left(\frac1{|p|^2}-\frac1{|q|^2}\right)+\frac{1}{|a_p|^2}\left(\frac1{|q|^2}-\frac1{|r|^2}\right)+\frac{1}{|a_q|^2}\left(\frac1{|r|^2}-\frac1{|p|^2}\right)\right)=0.
	\end{align*}	
    As a result, by the assumption $q\cdot p^\perp\neq0$, we establish the desired relation \eqref{eq triple}. This completes the proof of Proposition \ref{prop triple}.
	\end{proof}

	\vspace{0.6em}

	We end this subsection by establishing  Lemma \ref{lemma V}.  

	\begin{proof}[Proof of Lemma \ref{lemma V}]
	Throughout the proof, $\C^{|\Lambda_+|}$ is identified with
	$\R^{2|\Lambda_+|}$. Thus, for $z_k=x_k+iy_k$ and $\V_k^\Lambda=U_k^\Lambda+iW_k^\Lambda$, the real vector field
	associated with $\V^\Lambda$ has components $(U_k^\Lambda,W_k^\Lambda)_{k\in\Lambda_+}$.

    \vspace{0.3em}

	\noindent{\it Step 1: Infinitesimal stationarity of the finite marginal.} Let us define  $X_\Lambda(\Omega):=(\Omega_k)_{k\in\Lambda_+}$ and $F(\Omega)=f(X_\Lambda(\Omega))$ for $f\in C_c^1(\C^{|\Lambda_+|};\R)$. For
	$\Omega(t)=\Phi_t(\Omega_{in})$, $\Omega_{in}\in c_0^\alpha(\T^2)$, one has
	\begin{equation*}
	\Omega_k(t)-\Omega_k(0)
	=-\int_0^t B(\Omega(s))_kds,\quad t\in\R.
	\end{equation*}
    
	Then using Lemma \ref{lemma B} and conservation of enstrophy, for any $t\neq0$ and $k\in\Lambda_+$, we have
	\begin{align*}
	\left|\frac{\Omega_k(t)-\Omega_k(0)}{t}\right|
	\leq \frac1{|t|}\int_{0\wedge t}^{0\lor t}|B(\Omega(s))_k|ds\leq\frac{|k|}{4\pi^2|t|}
	\int_{0\wedge t}^{0\lor t}\|\Omega(s)\|_{L^2}^2ds=\frac{|k|}{4\pi^2}\|\Omega_{in}\|_{L^2}^2.
	\end{align*}
	Applying the mean-value theorem in $\R^{2|\Lambda_+|}$, we thus obtain
	\begin{align}\label{eq V2}
	\left|\frac{F(\Phi_t(\Omega_{in}))-F(\Omega_{in})}{t}\right|\leq \|\nabla f\|_{L^\infty}\left(\sum_{k\in\Lambda_+}\left|\frac{\Omega_k(t)-\Omega_k(0)}{t}\right|^2\right)^{1/2}&\leq C_{f,\Lambda}\|\Omega_{in}\|_{L^2}^2.	
	\end{align}
	Meanwhile, note that this dominating function is integrable with respect to $\mu_a$:
	\begin{equation*}
	\int\|\Omega_{in}\|_{L^2}^2d\mu_a(\Omega_{in})=2(2\pi)^2\sum_{n\in\Z^2_*}|a_n|^2<\infty.
	\end{equation*}

	On the other hand, by the differentiability of every fixed Fourier mode, we compute that
	\begin{align*}
	\left.\frac{d}{dt}F(\Phi_t(\Omega_{in}))\right|_{t=0}=\sum_{k\in\Lambda_+}\left[\partial_{x_k}f(X_\Lambda(\Omega_{in}))\operatorname{Re}(-B(\Omega_{in})_k)+\partial_{y_k}f(X_\Lambda(\Omega_{in}))\operatorname{Im}(-B(\Omega_{in})_k)\right].
	\end{align*}
	Recall that $\mu_a$ is invariant, which implies 
	\begin{equation*}
	\int F(\Phi_t(\Omega_{in}))d\mu_a(\Omega_{in})=\int F(\Omega_{in})d\mu_a(\Omega_{in}).
	\end{equation*}
    
	Taking \eqref{eq V2} into account, we may therefore differentiate at $t=0$ under the integral sign. Noting that $\nabla f(X_\Lambda)$ is $\mathcal F_\Lambda$-measurable and using Lemma
	\ref{lemma B2}, it follows that
	\begin{align}
	0&=\int DF(\Omega_{in})[-B(\Omega_{in})]d\mu_a(\Omega_{in})\notag\\
    &=\E\left[
	\nabla f(X_\Lambda)\cdot\E\left[\left(-B(\Omega_{in})_k\right)_{k\in\Lambda_+}
	|\mathcal F_\Lambda
	\right]\right]\notag\\
	&=\E\left[\nabla f(X_\Lambda)\cdot\V^\Lambda(X_\Lambda)\right]\notag\\    &=\int_{\C^{|\Lambda_+|}}\V^\Lambda(z)\cdot\nabla f(z)\rho_\Lambda(z)dz,
	\label{eq V3}
	\end{align}
	where the middle two dot products are understood in the corresponding
	real coordinates. 

	\vspace{0.6em}

	\noindent{\it Step 2: Liouville property of the Galerkin vector field.}
	Fix $k\in\Lambda_+$. We first show that $U_k^\Lambda$ and $W_k^\Lambda$ are independent of $x_k,y_k$. Recall that
	\begin{equation*}
	\V_k^\Lambda(z)=-\sum_{\substack{m+n=k,\;m,n\in\Lambda}}\kappa(m,n)z_mz_n.
	\end{equation*}
	The variables $x_k,y_k$ occur only in $z_k=x_k+iy_k$ and $z_{-k}=x_k-iy_k$. If $m=k$ or $n=k$, the relation $m+n=k$ forces the other index to be $0$, which does not belong to $\Lambda$. If $m=-k$ or $n=-k$, the other index must be $2k$. However, the corresponding coefficient vanishes because $-k$ and $2k$ are collinear:
	\begin{equation*}
	\kappa(-k,2k)=\kappa(2k,-k)=0.
	\end{equation*}
	Thus, every nonzero term in $\V_k^\Lambda$ is independent of $x_k,y_k$. Consequently, $\partial_{x_k}U_k^\Lambda=\partial_{y_k}W_k^\Lambda=0$. Summing over $k\in\Lambda_+$, we then derive 
	\begin{equation}\label{eq V4}	\dvg_{\R^{2|\Lambda_+|}}\V^\Lambda=\sum_{k\in\Lambda_+}\left(\frac{\partial U_k^\Lambda}{\partial x_k}+\frac{\partial W_k^\Lambda}{\partial y_k}\right)=0.
	\end{equation}

	\vspace{0.3em}

	\noindent{\it Step 3: The Gaussian density identity.}
	We now take $f\in C_c^\infty(\C^{|\Lambda_+|};\R)$ in
	\eqref{eq V3} and apply integration by parts, which implies that
	\begin{align*}
	0=\int_{\C^{|\Lambda_+|}}\V^\Lambda\cdot\nabla f\rho_\Lambda dz=-\int_{\C^{|\Lambda_+|}}f\dvg(\rho_\Lambda\V^\Lambda)dz.
	\end{align*}
	Hence
	\begin{equation*}
	\dvg(\rho_\Lambda\V^\Lambda)=0\quad \text{in }\mathcal{D}'(\R^{2|\Lambda_+|})
	\end{equation*}
	As $\rho_\Lambda\V^\Lambda$ is smooth, this identity thus holds pointwise. Additionally, in view of the definition of Gaussian density $\rho_\Lambda$, one has $\nabla\rho_\Lambda=-\rho_\Lambda\nabla Q_\Lambda$ and $\rho_\Lambda(z)>0$ for any
	$z\in\C^{|\Lambda_+|}$.

    In particular, combining these relations with \eqref{eq V4}, we obtain that
	\begin{align*}	0=\dvg(\rho_\Lambda\V^\Lambda)=\rho_\Lambda\dvg\V^\Lambda+\nabla\rho_\Lambda\cdot\V^\Lambda=-\rho_\Lambda\nabla Q_\Lambda\cdot\V^\Lambda,
	\end{align*}
	which implies the desired relation \eqref{eq V}. The proof of Lemma \ref{lemma V} is completed.    
	\end{proof}

	\subsection{Proof of the rigidity property}\label{Sec 3.3}

    We now combine the closure property and the triple relation to prove Lemma \ref{lemma C} and Proposition \ref{prop mu}. Roughly speaking, starting from one non-degenerate pair, the former generates an infinite set of active modes, while the latter propagates a single affine law for the inverse variances throughout that set.

    \begin{proof}[Proof of Lemma \ref{lemma C}]
        As the sets $\mathcal C_j(p,q)$ are increasing, for any $m,n\in\mathcal C(p,q)$, there exists $J\in\N^+$ such that $m,n\in\mathcal C_J(p,q)$. If $(m,n)$ is non-degenerate, then
	\begin{equation*}
	\pm(m+n)\in\mathcal C_{J+1}(p,q)\subset\mathcal C(p,q).
	\end{equation*}
	Thus $\mathcal C(p,q)$ is itself closed under non-degenerate sums. 
 
	Next we show that $\mathcal C(p,q)$ is an infinite set. Suppose now that $\mathcal C(p,q)$ were finite, and set	$R=\max_{n\in\mathcal C(p,q)}|n|$. 
	Choose $n\in\mathcal C(p,q)$ with $|n|=R$. Recall that the initial pair $p,q$ is noncollinear and has unequal lengths, which ensures that the generated set is neither collinear nor contained in a single centered circle. Hence there exists $m\in\mathcal C(p,q)$ not collinear with $n$. If $|m|<R$, replace $m$ by $-m$ so that $n\cdot m\geq0$. Closure under non-degenerate sums would give $n+m\in\mathcal C(p,q)$ and
	\begin{equation*}
	|n+m|^2=R^2+|m|^2+2n\cdot m>R^2,
	\end{equation*}
	which gives a contradiction. Thus every mode not collinear with $n$ has length $R$.

	On the other hand, as the set is not contained in one centered circle, there exists $\ell\in\mathcal C(p,q)$ with $|\ell|<R$. Then the preceding argument forces $\ell$ to be collinear with $n$. Choose $m$ not collinear with $n$, so $|m|=R$, and change its sign such that $\ell\cdot m\geq0$. Then $(\ell,m)$ is non-degenerate, while closure gives $\ell+m\in\mathcal C(p,q)$ and
	\begin{equation*}
	|\ell+m|^2=|\ell|^2+R^2+2\ell\cdot m>R^2,
	\end{equation*}
	again a contradiction. Therefore, we show that $\mathcal C(p,q)$ is infinite.  
    \end{proof}

	\begin{proof}[Proof of Proposition \ref{prop mu}]  

	Since $|p|\neq|q|$, there exist constants $c_0,c_1$ satisfying
	\begin{equation*}
	|a_p|^{-2}=c_0+c_1|p|^{-2},\quad|a_q|^{-2}=c_0+c_1|q|^{-2}.
	\end{equation*}
	Indeed, $c_0,c_1$ can be directly calculated by 
    \begin{equation*}
	c_0=\frac{|p|^{-2}|a_q|^{-2}-|q|^{-2}|a_p|^{-2}}{|p|^{-2}-|q|^{-2}},\quad c_1=\frac{|a_p|^{-2}-|a_q|^{-2}}{|p|^{-2}-|q|^{-2}}.
	\end{equation*}
    Using the evenness of $a$, one sees that \eqref{eq mu} holds on $\mathcal C_0(p,q)=\{\pm p,\pm q\}$.

	We now apply an induction on $j$. Note that By Proposition~\ref{prop clo}, $\mathcal C(p,q)\subset\mathcal S(a)$. Suppose \eqref{eq mu} is satisfied on $\mathcal C_j(p,q)$, and take $k\in\mathcal C_{j+1}(p,q)\setminus\mathcal C_j(p,q)$. After changing the sign of $k$ if necessary, we may write
	$k=m+n$ with $m,n\in\mathcal C_j(p,q)$ forming a non-degenerate pair. Then we have
    \begin{equation*}
        m+n+(-k)=0,\quad n\cdot m^{\perp}\neq 0.
    \end{equation*}
    Therefore, by Proposition~\ref{prop triple}, it follows that 
    \begin{equation}\label{eq mu3}
	\frac{1}{|a_{-k}|^2}\left(\frac1{|m|^2}-\frac1{|n|^2}\right)+\frac{1}{|a_m|^2}\left(\frac1{|n|^2}-\frac1{|k|^2}\right)+\frac{1}{|a_n|^2}\left(\frac1{|k|^2}-\frac1{|m|^2}\right)=0,
	\end{equation} 
    
    Meanwhile, using the induction hypothesis, one has
	\begin{equation*}
	\frac{1}{|a_m|^2}=c_0+\frac{c_1}{|m|^2},\quad\frac{1}{|a_n|^2}=c_0+\frac{c_1}{|n|^2}.
	\end{equation*}
	Moreover, we directly compute that
	\begin{align*}
	\left(c_0+\frac{c_1}{|m|^2}\right)\left(\frac{1}{|n|^2}-\frac{1}{|k|^2}\right)+\left(c_0+\frac{c_1}{|n|^2}\right)\left(\frac{1}{|k|^2}-\frac{1}{|m|^2}\right)=\left(c_0+\frac{c_1}{|k|^2}\right)\left(\frac{1}{|n|^2}-\frac{1}{|m|^2}\right).
	\end{align*}
	Substituting these identities into \eqref{eq mu3}, we obtain
	\begin{equation*}
	\left(\frac{1}{|m|^2}-\frac{1}{|n|^2}\right)\left(\frac{1}{|a_{-k}|^2}-c_0-\frac{c_1}{|k|^2}\right)=0.
	\end{equation*}
	Since $|m|\neq|n|$, the first factor is nonzero, which thus implies that
	\begin{equation*}
	|a_{k}|^{-2}=|a_{-k}|^{-2}=c_0+\frac{c_1}{|k|^2}.
	\end{equation*}
	
    Thus \eqref{eq mu} holds on $\mathcal C_{j+1}(p,q)$. This completes the induction and the proof of Proposition \ref{prop mu}.

	\end{proof}

	\section{Classification of invariant Gaussian measures}\label{Sec 4}

    With the geometric structure of Section \ref{Sec 3} at hand, we are now ready to complete the classification. Indeed, Proposition \ref{prop mu} rules out every non-degenerate active pair for an invariant measure, leaving precisely the collinear and co-circular support configurations. For finite active support, the quasi-invariant classification follows from the same closure mechanism.

	\subsection{Proof of Theorem \ref{thm 1}}

	\begin{proof}[Proof of Theorem \ref{thm 1}] The proof is divided into four parts, establishing successively 
    \begin{equation*}
        \textup{(i)}\Rightarrow\textup{(ii)}\Rightarrow\textup{(iii)}\Rightarrow\textup{(iv)}\Rightarrow\textup{(i)}.
    \end{equation*}
     
	\noindent\textup{(i)}$\Rightarrow$\textup{(ii)}.\quad  Suppose that $\mu_a$ is invariant and that $\mathcal S(a)$ contains a non-degenerate pair $p,q$. By Lemma \ref{lemma C} and  Proposition~\ref{prop clo}, $\mathcal C(p,q)\subset\mathcal S(a)$ is infinite. Moreover, using Proposition \ref{prop mu}, there exist constants $c_0,c_1\in\R$ such that
	\begin{equation}\label{eq thm1}
	|a_n|^{-2}=c_0+\frac{c_1}{|n|^2}\quad\forall\,n\in\mathcal C(p,q).
	\end{equation}
    
	Since every bounded subset of $\Z^2_*$ is finite, we may choose distinct $\{n_j\}_{j\in\N^+}\subset\mathcal C(p,q)$ such that $|n_j|\rightarrow\infty$. On the other hand, recall that $a\in h^\sigma\subset\ell^2(\Z^2_*)$, which implies $a_{n_j}\rightarrow0$.

	We now consider all possible signs of $c_0$. If $c_0>0$, then \eqref{eq thm1} gives $|a_{n_j}|^2\rightarrow c_0^{-1}>0$, which contradicts $a_{n_j}\rightarrow0$. If $c_0=0$, the positivity of $|a_{n_j}|^{-2}$ implies $c_1>0$, and hence $|a_{n_j}|^2=|n_j|^2/c_1\rightarrow\infty$, which is again a contradiction. Finally, if $c_0<0$, then the right-hand side of \eqref{eq thm1} is negative for all sufficiently large $j$, contradicting $|a_{n_j}|^{-2}>0$. 
    
    Consequently, $\mathcal S(a)$ contains no non-degenerate pair. Hence \textup{(i)} implies \textup{(ii)}.

	\vspace{0.6em}
	\noindent\textup{(ii)}$\Rightarrow$\textup{(iii)}.\quad   If all points in $\mathcal S(a)$ are pairwise collinear, then the active Fourier support is contained in a line through the origin. Otherwise, we may choose noncollinear $p,q\in\mathcal S(a)$. Since $(p,q)$ is degenerate, one has $|p|=|q|=:R$. For any $r\in\mathcal S(a)$, if $r$ is not collinear with $p$, the degeneracy of $(r,p)$ gives $|r|=R$. If $r$ is collinear with $p$, then $r$ is not collinear with $q$, and the degeneracy of $(r,q)$ again yields $|r|=R$. Therefore, $\mathcal S(a)$ is contained in the circle centered at the origin with radius $R$, which proves \textup{(iii)}.

	\vspace{0.6em}
	\noindent\textup{(iii)}$\Rightarrow$\textup{(iv)}.\quad  Fix $\Omega_{in}\in c_0^\alpha(\T^2)$ whose Fourier coefficients vanish outside $\mathcal S(a)$. This property holds for $\mu_a$-almost every initial datum. Suppose first that $\mathcal S(a)$ is contained in a line through the origin. For $N\in\N^+$, set
	\begin{equation*}
    \Omega_{in,N}:=\sum_{\substack{n\in\Z^2_*,\;|n|\leq N}}(\Omega_{in})_ne^{in\cdot x}.
	\end{equation*}
	Every pair of active modes is collinear, and hence $B(\Omega_{in,N})=0$ for every $N\in\N^+$. Then using Lemma \ref{lemma B}  and $\Omega_{in,N}\rightarrow\Omega_{in}$ in $L^2(\T^2)$, one has
	\begin{equation*}	B(\Omega_{in})_k=0\quad\forall\,k\in\Z^2_*.
	\end{equation*}
	Moreover, $B(\Omega_{in})=\nabla\cdot(U[\Omega_{in}]\Omega_{in})$ is well-defined in $\mathcal D'(\T^2)$ because $U[\Omega_{in}],\Omega_{in}\in L^2(\T^2)$. Since all its Fourier coefficients vanish, one has $B(\Omega_{in})=0$ in $\mathcal D'(\T^2)$. Therefore, the constant trajectory $\Omega(t)=\Omega_{in}$ solves equation \eqref{eq euler}, and uniqueness in $ c_0^\alpha(\T^2)$ implies $\Phi_t(\Omega_{in})=\Omega_{in}$ for any $t\in\R$.

	Suppose now that $\mathcal S(a)$ is contained in a circle centered at the origin with radius $R$. This active Fourier support is finite, and
	\begin{equation*}
	(-\Delta)^{-1}\Omega_{in}=R^{-2}\Omega_{in}.
	\end{equation*}
	It follows that
	\begin{equation*}
	u=R^{-2}\nabla^\perp\Omega_{in},\quad u\cdot\nabla\Omega_{in}=R^{-2}\nabla^\perp\Omega_{in}\cdot\nabla\Omega_{in}=0,
	\end{equation*}
    which again implies that  $B(\Omega_{in})=0$ in $\mathcal D'(\T^2)$. Thus, by the same argument, one has $\Phi_t(\Omega_{in})=\Omega_{in}$ for any $t\in\R$.   This proves \textup{(iv)}.

    \vspace{0.6em}
	
    \noindent\textup{(iv)}$\Rightarrow$\textup{(i)}.\quad  By condition \textup{(iv)}, there exists a Borel set $E\subset c_0^\alpha(\T^2)$ with $\mu_a(E)=1$ such that
	\begin{equation*}
	\Phi_t(\Omega_{in})=\Omega_{in}\quad\forall\,\Omega_{in}\in E,\;t\in\R.
	\end{equation*}
	Hence, for every Borel set $A\subset c_0^\alpha(\T^2)$,
	\begin{equation*}
	(\Phi_t)_*\mu_a(A)=\mu_a\left(\Phi_t^{-1}(A)\right)=\mu_a(A).
	\end{equation*}
    
	This proves \textup{(i)}.  The proof of Theorem \ref{thm 1} is thus completed.
	\end{proof}

    \subsection{Proof of Theorem \ref{thm 2}}
    
    \begin{proof}[Proof of Theorem \ref{thm 2}]
    It suffices to show \textup{(i)}$\Rightarrow$\textup{(ii)}. Assume that $\mu_a$ is quasi-invariant and fix a sequence $\{t_j\}_{j\in\N^+}$ such that $t_j\rightarrow0$ and $t_j\neq0$. In view of Definition~\ref{def 1}, one has
    \begin{equation*}
(\Phi_{t_j})_*\mu_a\ll\mu_a
\quad\forall\,j\in\N^+.
\end{equation*}
    By Proposition~\ref{prop clo}, $\mathcal S(a)$ is closed under the sum and difference of every non-degenerate pair.

    Suppose for contradiction that $\mathcal S(a)$ contains a non-degenerate pair $p,q$. By Lemma~\ref{lemma C} and Proposition~\ref{prop clo},
    \begin{equation*}
    \mathcal C(p,q)\subset\mathcal S(a)\quad\text{and}\quad\mathcal C(p,q)\text{ is infinite}.
    \end{equation*}
    By our assumption, $\mathcal S(a)$ is contained in a finite union of lines through the origin. Therefore, there exists a line $\R m$ such that $\mathcal C(p,q)\cap \R m$ is infinite. We may therefore choose distinct modes $\{n_\ell\}_{\ell\in\N^+}\subset\mathcal C(p,q)\cap \R m$ such that $|n_\ell|\rightarrow\infty$.

    Since $p$ and $q$ are noncollinear, $\mathcal C(p,q)$ is not contained in $\R m$. Hence we may choose $k\in\mathcal C(p,q)\setminus \R m$. For all sufficiently large $\ell$, the pair $(n_\ell,k)$ is non-degenerate. Then by Proposition~\ref{prop clo}, one has $n_l+k\in\mathcal{S}(a)$.
    
    We now write $n_\ell=\lambda_\ell m$, where the $\lambda_\ell$ are distinct. For $\ell\neq\ell'$, we have
    \begin{equation*}
    (n_\ell+k)\cdot(n_{\ell'}+k)^\perp=(\lambda_\ell-\lambda_{\ell'})m\cdot k^\perp\neq0.
    \end{equation*}
    Thus the modes $\{n_\ell+k\}_{l\in\N^+}$ lie on infinitely many distinct lines through the origin, contradicting the assumption on $\mathcal S(a)$ that $\mathscr{D}(a)=\{\R m:m\in \mathcal{S}(a)\}$ is finite.

    Consequently, $\mathcal S(a)$ contains no non-degenerate pair. By the implication \textup{(ii)}$\Rightarrow$\textup{(i)} in Theorem~\ref{thm 1}, we thus show that $\mu_a$ is invariant. This completes the proof.
    \end{proof}
	\normalem
	\bibliographystyle{abbrv}
	\bibliography{euler_ref}

	\end{document}